\documentclass[12pt,reqno]{amsart}

\usepackage[dvips]{graphicx} 
\usepackage[arrow,matrix,curve]{xy}

\usepackage{amssymb, latexsym, amsmath, amscd, array, 
}

\newtheorem{theorem}{Theorem}[section]

\newtheorem{proposition}[theorem]{Proposition}

\theoremstyle{definition}

\numberwithin{equation}{section}
\numberwithin{figure}{section} 
\numberwithin{table}{section}

\DeclareMathOperator{\PD}{{\rm PD}}

\newcommand\st{{\rm st}}

\newcommand\dist{{\rm dist}}

\newcommand \cqfd{\unskip\kern 6pt\penalty 500
\raise -2pt\hbox{\vrule\vbox to10pt{\hrule width 4pt
\vfill\hrule}\vrule}\par}                 

\def\adots{\mathinner{\mkern2mu\raise1pt\hbox{.}
\mkern3mu\raise4pt\hbox{.}\mkern1mu\raise7pt\hbox{.}}}
\def\hfl#1{\frac{\buildrel{#1}}{{\hbox to 12mm{\rightarrowfill}}}}
  
\def \\R^n \times \R^n
\rightarrow \R{\mathop{\R^n \times \R^n
\rightarrow \R}}

 \newcommand\R {{\mathbb R}}
\newcommand\RR {{\mathbb R}} \newcommand\RP {{\mathbb R}{\mathbb P}}
 \newcommand\Z {{\mathbb Z}}

\newcommand{\smat}[4]
{{\(\!\!\begin{array}{cc}{#1}\!&\!{#2}\\begin{equation*}-0.1cm]{#3}\!&\!{#4}\end{array}\!\!\)}}

\long\def\forget#1\forgotten{} %

\long\def\forgett#1\forgottent{} %

\def\circ{\mathchoice%
 {\mathrel{\raise 1pt\hbox{$\scriptstyle\mathchar"020E$}}}
 {\mathrel{\raise 1pt\hbox{$\scriptstyle\mathchar"020E$}}}
 {\mathrel{\raise 1pt\hbox{$\scriptscriptstyle\mathchar"020E$}}}
 {}
}

\newcommand{\nc}{\newcommand} \nc{\on}{\operatorname}

\nc{\df}{\on{\it df}}

\nc{\conf}{\on{conf}}

\nc{\spt}{\on{spt}}
\nc{\norm}[1]{\| #1 \|}
\nc{\parallelleer}{\norm{\ }} 
\nc{\parallelh}{\norm h} 
\nc{\parallelk}{\norm k} 
\nc{\parallelx}{\norm x} 
\nc{\parallelhrr}{\norm {h_\RR}} 
\nc{\parallelom}{\norm \omega} 
\nc{\parallelomij}{\norm {\omega_{i_j}}} 
\nc{\parallelomx}{\norm {\omega_{x}}} 
\nc{\parallelpi}{\norm \pi} 
\nc{\parallelalf}{\norm \alpha} 
\nc{\parallelalfs}{\norm {\alpha_s}} 
\nc{\parallelalfi}{\norm {\alpha_i}} 
\nc{\parallelalfij}{\norm {\alpha_{i_j}}} 
\nc{\parallelbeta}{\norm \beta} 
\nc{\parallelbetat}{\norm {\beta_t}} 
\nc{\parallelhcapalf}{\norm {h \cap \alpha}} 
\nc{\parallelPDralf}{\norm {\PD_\RR(\alpha)}} 
\nc{\strichleer}{| \  |}
\nc{\NN}{\mathbb N}

\nc{\rr}{\mbox{$\scriptstyle\mathbb R$}}

\nc{\dF}{{\it dF}} 

\nc{\DF}{{\it DF}} 

\nc{\ds}{{\it ds}} 

\nc{\dvol}{{\it dvol}}

\nc{\grad}{{\rm grad}} 
\nc{\strichw}{\|\omega\|} 
\nc{\strichwx}{|\omega_x|}
\nc{\Hess}{{\rm Hess}}

\begin{document}

\title{Bi-Lipschitz approximation by finite-dimensional imbeddings}

\author{Karin Usadi Katz}

\author[M.~Katz]{Mikhail G. Katz$^{*}$}

\address{Department of Mathematics, Bar Ilan University, Ramat Gan
52900 Israel} \email{\{katzmik\}@macs.biu.ac.il (remove curly braces)}

\thanks{$^{*}$Supported by the Israel Science Foundation (grants
no.~84/03 and 1294/06) and the BSF (grant 2006393)}



\date{\today}

\begin{abstract}
We show that the Kuratowski imbedding of a Riemannian manifold
in~$L^\infty$, exploited in Gromov's proof of the systolic inequality
for essential manifolds, admits an approximation by a
$(1+C)$--bi-Lipschitz (onto its image), finite-dimensional imbedding
for every~$C>0$.  Our key tool is the first variation formula thought
of as a real statement in first-order logic, in the context of
non-standard analysis.
\end{abstract}

\maketitle 

\tableofcontents

\section{Metric imbeddings and Gromov's theorem}

In '83, M. Gromov proved that the least length ({\em systole}, denoted
``sys'') of a non-contractible loop in a closed Riemannian
manifold~$M$ is bounded above in terms of the volume of~$M$, if~$M$
satisfies the topological hypothesis of being {\em essential} (for
instance, if~$M$ is aspherical).

A key technique in Gromov's seminal text \cite{Gr1} is the Kuratowski
imbedding.  Namely, Gromov imbeds a Riemannian manifold~$M$ into the
space
\[
L^\infty = L^\infty(M)
\]
of bounded Borel functions on~$M$.  Here a point~$x\in M$ is sent to
the function~$f_x$ defined by
\begin{equation}
f_x(y) = \dist(x,y) \quad \forall y\in X,
\end{equation}
where ``$\dist$'' is the Riemannian distance function in~$M$.  This
imbedding is strongly isometric, in the sense that the intrinsic
distance in~$M$ coincides with the ambient distance in~$L^\infty$
defined by the sup-norm.

The fact that the space~$L^\infty(M)$ is infinite-dimensional may have
given some readers of \cite{Gr1} the impression that
infinite-dimensionality of the imbedding is an essential aspect of
Gromov's proof of the systolic inequality for essential manifolds.  In
fact, this is not the case. Indeed, we can choose a
maximal~$\epsilon$-separated net~${\mathcal M} \subset M$
with~$|{\mathcal M}| < \infty$ points (by compactness of~$M$, every
infinite set would have an accumulation point, contradicting
$\epsilon$-separation).

Choose~$\epsilon$ satisfying~$\epsilon < \frac{1}{10} {\rm sys} (M)$.
Consider the reulting imbedding
\begin{equation}
\label{23}
M \to \ell^\infty({\mathcal M})
\end{equation}
by the distance functions from points of~${\mathcal M}$.  Then, for
the metric inherited from the imbedding, the systole goes down by a
factor at most~$5$, see \cite[p.~97]{SGT}.  Thus the systolic problem
can easily be reduced to finite-dimensional imbeddings.

In the present text, we show that, similarly, by choosing a
sufficiently fine~$\epsilon$-net, one can force the map \eqref{23} to
be~$(1+C)$--bi-Lipschitz onto its image, for all~$C>0$ (see
Theorem~\ref{NSA} below):

\begin{theorem}
\label{11}
Let~$M$ be a compact Riemannian manifold without boundary.  For every
$C>0$, there exists a~$(1+C)$--bi-Lipschitz finite-dimensional
imbedding of~$M$, approximating its isometric imbedding
in~$L^{\infty}(M)$.
\end{theorem}

Here a homeomorphism~$\phi$ is called~$K$--bi-Lipschitz if
\[
\dist(\phi(x),\phi(y))<K\dist(x,y)
\]
for all~$x,y$, and similarly for the inverse~$\phi^{-1}$.

It follows that finite-dimensional approximations work well for the
filling radius inequality, as well, namely the inequality relating the
filling radius of~$M$ and the volume of~$M$ (see~\cite{Gr1}).

The bi-Lipschitz property was discussed in \cite[p.~115]{Gu2}, where a
``sketch'' of a proof concludes as follows: ``Finally, we can
generalize the trigonometry argument to almost flat manifolds using
the Toponogov comparison theorem''.  In fact, we will see that both
``almost flatness'' and ``Toponogov's theorem'' miss the mark
somewhat, as the relevant ingredient in the proof is the first
variation formula, which can be applied in the absence of curvature
hypotheses, and does not require the difficult (albeit classical)
result of Toponogov.  (Similarly, even in the flat case, the argument
sketched in \cite{Gu2} may contain a gap in the case when, in the
notation of \cite{Gu2}, the pair~$x,y$ are much closer than the scale
of the~$\delta$-net, as even a quadratic estimate
on~$d(x,x_i)-d(y,x_i)$ may still be greater than~$d(x,y)$.)

Our method of proof involves the following technique.  We use the {\em
tranfer principle\/} of non-standard analysis (see Section~\ref{ten},
item \ref{101}) to conclude that the first variation formula
\eqref{22} must apply also to the non-standard line through a pair of
infinitely close hyperreal points.  
The main idea is to view the first variatiom formula from differential
geometry, as a statement in first-order logic.

Note that such concepts as the injectivity radius and the first
variation formula can be formulated in first order logic.  This is
essential for our argument, since the transfer principle allows one to
conclude that real statements are true over~$\R^*$ just as they are
true over~$\R$, only if such statements are in first-order logic,
i.e.~quantification over elements is allowed, quantification over sets
or sequences is not allowed.  

The finite-dimensional approximation is used in an analytic proof of
Gromov's systolic inequality in \cite{AK}.

Section~\ref{two} reviews the basic differential geometric notions
used in our proof.  Section~\ref{three} defines the ingredients of the
proof of our approximation result.  Section~\ref{blowup} discusses the
real blow-up of~$M\times M$ along the diagonal, used in the proof of
the main Theorem~\ref{11}.  Section~\ref{four} contains the hyperreal
part of the proof, which starts with a choice of a {\em
hyperinteger\/} (see Section~\ref{ten}, item~\ref{108}).
Section~\ref{ten} outlines the basic principles of non-standard
analysis.

\section{Geodesic equation, injectivity radius, and first variation}
\label{two}

A smooth curve~$\alpha(s)$ in a complete~$n$-dimensional manifold~$M$
is a geodesic if for each~$k=1,2, \ldots, n$, we have in coordinates
\begin{equation}
\label{38proto}
(\alpha^k)^{''} + \Gamma^k_{ij}(\alpha^i)^{'}(\alpha^j)^{'}
 = 0 \quad \hbox{where}
\quad ^{'} = {d \over ds}\ ,
\end{equation}
meaning that
\[
(\forall k) \quad {d^2 \alpha^k \over ds^2} + \Gamma^k_{ij} {d\alpha^i
\over ds} {d\alpha^j \over ds} = 0,
\]
The symbols~$\Gamma_{ij}^k$ can be expressed in terms of the first
fundamental form and its derivatives as follows :
$$\Gamma^k_{ij} = {1 \over 2}(g_{i\ell;j} - g_{ij;\ell} + g_{j\ell;i})
g^{\ell k},$$ where~$g^{ij}$ is the inverse matrix of~$g_{ij}$.
Denote by
\begin{equation}
\label{geo}
\gamma(s)= \gamma(p,v,s)
\end{equation}
the geodesic starting at~$p=\gamma(0)$, with initial
vector~$v=\gamma'(0)$.  We have a well-known homogeneity property
\[
\gamma(x,tv,s) = \gamma(x,v,ts)
\]
for all real~$t$.  We define the exponential map 
\[
\exp_p: T_p M \to M
\]
by~$v \mapsto \gamma(p,v,1)$.

The injectivity radius~${\rm InjRad}_p(M)$ of~$M$ at~$p$ is the
supremum of all~$r$ such that the exponential map is injective on a
ball of radius~$r$ centered at the origin of~$T_p M$.  The global
injectivity radius of~$M$ is defined by minimizing~${\rm InjRad}_p(M)$
over~$p$.

The formula relating the following pair of metric quantities:
\begin{enumerate}
\item
the distance~$u(s)$ from a point~$q\in M$ to~$\gamma(p,v,s)$
(where~$v$ is a unit vector), realized by a geodesic joining them
(which is assumed to be minimizing);
\item
the angle~$\alpha$ at~$p$ formed by the two geodesics,
\end{enumerate}
is called the first variation formula:
\begin{equation}
\label{22}
u'(0) = - \cos \alpha.
\end{equation}

\section{Approximation by finite-dimensional imbeddings}
\label{three}

\begin{theorem}
\label{NSA}
Let~$M$ be a compact Riemannian manifold without boundary.  For every
$C>0$, there exists a~$(1+C)$--bi-Lipschitz finite-dimensional
imbedding of~$M$, approximating its isometric imbedding
in~$L^{\infty}(M)$.
\end{theorem}

\begin{proof}
For each~$n\in {\mathbb N}$, choose a maximal~$\frac{1}{n}$-separated
net
\[
{\mathcal M}_n\subset M,
\]
and imbed the manifold~$M$ in~$\ell^\infty$ by the collection of
distance functions from the points in the net, namely, by a map
\begin{equation}
\label{42}
\iota_n : M \to \ell^\infty({\mathcal M}_n).
\end{equation}
If there exists a real~$C>0$ such that the imbedding is
not~$(1-C)$--bi-Lipschitz, then there is a pair of points~$x_n,y_n\in
M$ such that the distance~$d(x_n, y_n)$ satisfies
\begin{equation}
\label{911}
| \iota_n(x) - \iota_n(y) | \leq (1-C) d(x_n, y_n),
\end{equation}
meaning that
\begin{equation}
\label{43b}
| d(x_n, z_n) - d(y_n, z_n) | \leq (1-C) d(x_n, y_n)
\end{equation}
for every~$z_n \in {\mathcal M}_n$.  Let~$\gamma_n(s)$ be the geodesic
parametrized by arclength starting at~$x_n=\gamma_n(0)$, passing
through~$y_n$.  Let~$q_n= \gamma_n(b)$ where
\[
b= \frac{1}{2}{\rm InjRad}(M).
\]
Let~$a_n \in {\mathcal M}_n$ be a point of the maximal net nearest
to~$q_n$.  Let~$\alpha_n$ be the angle at~$x_n$:
\[
\alpha_n = \angle a_n x_n y_n.
\]
The idea is to show that choosing a sufficiently fine net will force
the angle to be small.  Define a function~$u_n= u_n(s)$ by setting
\[
u_n(s) = d(\gamma_n(s), a_n).
\]
Then we have the first variation formula
\begin{equation}
\label{fv}
u'_n(0) = - \cos \alpha_n.
\end{equation}
Let also
\[
v_n=\gamma_n'(0) \in T_{x_n}M
\]
be its initial vector, for which we will use the briefer
notation~$(x_n, v_n)$.

Thus we obtain a sequence of finite-dimensional imbeddings~$\iota_n$
as in~\eqref{42}.  We will argue by contradiction.  Suppose for
each~$n$ we can find a pair~$(x_n, y_n)$ satisfying \eqref{911}.  We
assume without loss of generality that~$d(x_n, y_n)$ is smaller than
the injectivity radius of~$M$.  By the compactness of the unit tangent
sphere bundle of~$M$, we can replace the sequence~$(x_n, v_n), n\in
{\mathbb N}$ by a convergent subsequence.  Let
\[
(p,v)= \lim_{n\to \infty} (x_n, v_n),
\]
and let~$\gamma(t)$ be the unique geodesic with initial data~$(p,v)$.
Let~$q=\gamma(b)$, where~$b =\frac{1}{2} {\rm InjRad}(M)$, as in
Figure~\ref{NSAfigure}.  The proof is completed by a hyperreal
technique in Section~\ref{four}. 
\end{proof}

\section{Real blow-up along the diagonal}
\label{blowup}

To handle a technical point in the proof of Theorem~\ref{NSA}, will
will need the following auxiliary construction.  Consider the product
manifold~$M^{\times 2}= M \times M$, and the diagonal~$D\subset
M^{\times 2}$.  We consider the real blow-up~$\hat M_D^{\times 2}$
of~$M^{\times 2}$ along~$D$:
\[ 
\beta: \hat M_D^{\times 2} \to M^{\times 2}.
\]
Here the inverse image of a point~$(x,x) \in D\subset M^{\times 2}$
under the map~$\beta$ is a copy of~$\RP^{n-1}$, thought of as the
collection of lines~$\ell$ orthogonal to~$D\subset M^{\times 2}$ at
the point~$(x,x) \in M^{\times 2}$.  Projecting to the second
component in~$M\times M$, one can think of~$\ell$ as a line in~$M$
passing through~$x\in M$.

We define a function 
\[
F: \hat{M}_D^{\times{2}}\times M \to \R
\]
on the product~$\hat{M}_D^{\times{2}}\times M$ as follows.  Away from
the diagonal~$D$, a point in~$\hat{M}_D^{\times{2}}\times M$ is
represented by a triple~$(x,y,z)$ of points of the manifold~$M$
itself, and we define~$f$ by setting
\[
F(x,y,z)= \frac{|d(x,z)-d(y,z)|}{d(x,y)}.
\]
For points of the form
\[
(x,v)\in \beta^{-1}(D), \quad D \subset M^{\times 2},
\]
where the unit vector~$v$ is tangent to a line~$\ell$ through~$x$, we
set
\[
F((x,v),z)= |u'(0)|,
\]
where~$u(s)=d(\gamma(s),z)$, and~$\gamma(s)=\gamma(x,v,s)$ is the
geodesic satisfying~$\gamma(0)=x$ and~$\gamma'(0)=v$ (see
\eqref{geo}).  In particular, we have
\begin{equation}
\label{41c}
F((x,v),z)= 1
\end{equation}
if $z$ lies on a minimizing geodesic $\gamma(x,v,s)$.
\begin{proposition}
\label{41b}
The function~$F$ is continuous in the region defined by~$d(x,z) \leq
\frac{1}{2}{\rm InjRad} M$.
\end{proposition}

\begin{proof}
Let~$(x_n,v_n)$ be a sequence converging to~$(p,v)$.  In view of the
first variation formula, to prove the continuity of~$F$, it suffices
to show that the angle~$\alpha_n$ converges to~$\alpha$, the angle
formed at~$p$ by~$v$ and~$\gamma'(0)$.  This is immediate from the
fact that the exponential map
\[
\exp_z: T_z M \to M
\]
at the point~$z\in M$ is a diffeomorphism onto its image around~$p$.
\end{proof}

\section{Choice of hyperinteger}
\label{four}

\begin{figure}
\includegraphics[height=2.7in]{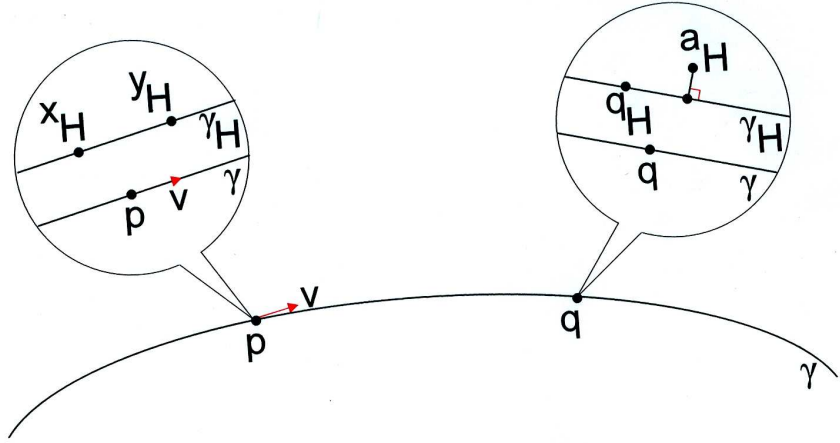}
\caption{Microscopic images of a pair of infinitely close lines
$\gamma$ and~$\gamma_H$ on~$M^*$}
\label{NSAfigure}
\end{figure}

We continue with the proof by contradiction of Theorem~\ref{NSA}.
Let~$H$ be an infinite Robinson hyperinteger (see Section~\ref{ten},
item~\ref{108}).  Then the sequence~$({\mathcal M}_n)$ is defined for
the value~$H$ of the index, by the extension principle (see
Section~\ref{ten}, item~\ref{101}).  Note that by compactness of~$M$,
we have
\begin{equation}
\label{41}
M=\st({\mathcal M}_H),
\end{equation}
where ``st'' is the standard part function (see Section~\ref{ten},
item~\ref{103}).  Since the relation~\eqref{911} is satisfied at all
finite values of the index~$n$, it is satisfied at the value~$H$, as
well, by the transfer principle (see Section~\ref{ten},
item~\ref{102}).  The points~$x_H$ and~$y_H$ are infinitely close
to~$p\in M$.  The geodesic~$\gamma_H$ passes through both~$x_H$
and~$y_H$ by construction, and is infinitely close to the limiting
geodesic~$\gamma$.  By the transfer principle,
\[
d(a_H, q_H) < \frac{1}{H}.
\]
Equation \eqref{911} yields
\begin{equation}
\label{911H}
| \iota_H(x) - \iota_H(y) | \leq (1-C) d(x_H, y_H).
\end{equation}
Let~$\Delta s = d(x_H, y_H)$, so that~$\gamma_H(\Delta s) = y_H$.
Just as for a finite value of the index, we have
\[
\alpha_H = \angle a_H x_H y_H.
\]
The point~$(x_H,y_H)$ is infinitely close to the point~$(p,v)\in \hat
M_D^{\times 2}$ of the blow-up constructed in Section~\ref{blowup}.
By Proposition~\ref{41b}, the function~$F$ is continuous.
Since~$(x_H,y_H,a_H)\approx ((p,v),q)$, we have
\begin{equation}
\label{53} 
F(x_H,y_H,a_H)\approx F((p,v),q).
\end{equation}
Therefore
\[
\left| \frac{\Delta u_H}{\Delta s} \right| = F(x_H,y_H,a_H)\approx
F((p,v),q) = 1,
\]
and therefore
\begin{equation}
\label{55}
\frac{\Delta u_H}{\Delta s} \approx -1.
\end{equation}
Note that by the transfer principle and the first variation
\eqref{fv}, we obtain
\begin{equation}
\label{54}
u_H'(0)= -\cos \alpha_H,
\end{equation}
but \eqref{55} is not immediate from \eqref{54}, as the function $u_H$
is only internal rather than standard, so that one cannot
apply~\eqref{approx} directly.  Equation~\eqref{55} is equivalent to
\[
\frac{d(\gamma_H(\Delta s), a_H) - d(x_H, a_H)}{\Delta s} \approx - 1
\]
or
\[
\frac{d(y_H, a_H) - d(x_H, a_H)}{d(x_H, y_H)} \approx - 1.
\]
Thus an application of the standard part function ``st'' (see
Section~\ref{ten}, item~\ref{103}) yields
\[
{\rm st} \left( \frac{d(x_H, a_H) - d(y_H, a_H)}{d(x_H, y_H)}
\right)=1,
\]
contradicting \eqref{911H}.  The resulting contradiction proves that
some finite-dimensional imbedding will necesarily
be~$(1-C)$--bi-Lipschitz, completing the proof of Theorem~\ref{NSA}.

\section{A non-standard glossary}
\label{ten}

The present section is included mainly for the benefit of the reader
not yet familiar with the general framework of non-standard analysis.
The section can be omitted, shortened, or retained as is, as per
recommendation of the referee.  

A popular introduction to the subject may be found in \cite{St},
chapter~6: ``Ghosts of departed quantities''.

In this section we present some illustrative terms and facts from
non-standard calculus \cite{Ke}.  The relation of being infinitely
close is denoted by the symbol~$\approx$.  Thus,~$x\approx y$ if and
only if~$x-y$ is infinitesimal.

\subsection{Natural hyperreal extension~$f^*$}
\label{101}

The construction of the hyperreals is carried out in the framework of
the standard axiomatisation of set theory, denoted ZFC.  Here ZFC
stands for the axiom system of Zermelo and Fraenkel, with the addition
of the Axiom of Choice.

The {\em extension principle\/} of non-standard calculus states that
every real function~$f$ has a hyperreal extension, denoted~$f^*$ and
called the natural extension of~$f$.  The {\em transfer principle\/}
of non-standard calculus asserts that every real statement true
for~$f$, is true also for~$f^*$.  For example, if~$f(x)>0$ for every
real~$x$ in its domain~$I$, then~$f^*(x)>0$ for every hyperreal~$x$ in
its domain~$I^*$.  Note that if the interval~$I$ is unbounded,
then~$I^*$ necessarily contains infinite hyperreals.  We will
typically drop the star~$^*$ so as not to overburden the notation.

\subsection{Internal set}
\label{102}
Internal set is the key tool in formulating the transfer principle,
which concerns the logical relation between the properties of the real
numbers~$\R$, and the properties of a larger field denoted
\[
\R^*
\]
called the {\em hyperreal line}.  The field~$\R^*$ includes, in
particular, infinitesimal (``infinitely small") numbers, providing a
rigorous mathematical realisation of a project initiated by Leibniz.
Roughly speaking, the idea is to express analysis over~$\R$ in a
suitable language of mathematical logic, and then point out that this
language applies equally well to~$\R^*$.  This turns out to be
possible because at the set-theoretic level, the propositions in such
a language are interpreted to apply only to internal sets rather than
to all sets.  Note that the term ``language" is used in a loose sense
in the above.  A more precise term is {\em theory in first-order
logic}.  Here a statement in first order logic by definition involves
quantification only over elements (quantification over sets or
sequences is not allowed).

Internal sets include natural extension of standard sets.

\subsection{Standard part function}
\label{103}
The standard part function ``st" is the key ingredient in Abraham
Robinson's resolution of the paradox of Leibniz's definition of the
derivative as the ratio of two infinitesimals
\[
\frac{dy}{dx}.
\]
The standard part function associates to a finite hyperreal
number~$x$, the standard real~$x_0$ infinitely close to it, so that we
can write
\begin{equation*}
\mathrm{st}(x)=x_0.
\end{equation*}
In other words, ``st'' strips away the infinitesimal part to produce
the standard real in the cluster.  The standard part function ``st" is
not defined by an internal set (see item~\ref{102} above) in
Robinson's theory.

\subsection{Cluster}
\label{104}
Each standard real is accompanied by a cluster of hyperreals
infinitely close to it.  The standard part function collapses the
entire cluster back to the standard real contained in it.  The cluster
of the real number~$0$ consists precisely of all the infinitesimals.
Every infinite hyperreal decomposes as a triple sum
\[
H+r+\epsilon,
\]
where~$H$ is a hyperinteger (see item~\ref{108} below), while~$r$ is a
real number in~$[0,1)$, and~$\epsilon$ is infinitesimal.
Varying~$\epsilon$ over all infinitesimals, one obtains the cluster
of~$H+r$.

\subsection{Derivative}
\label{105}
To define the derivative of~$f$ in this approach, one no longer needs
an infinite limiting process as in standard calculus.  Instead, one
sets
\begin{equation}
\label{deri}
f'(x) = \mathrm{st} \left( \frac{f(x+\epsilon)-f(x)}{\epsilon}
\right),
\end{equation}
where~$\epsilon$ is infinitesimal, yielding the standard real number
in the cluster of the hyperreal argument of ``st''.  Here the
derivative exists if and only if the value~\eqref{deri} is independent
of the choice of the infinitesimal.    Note that
\begin{equation}
\label{approx}
f'(x) \approx \frac{f(x+\epsilon)-f(x)}{\epsilon}.
\end{equation}
The addition of ``st'' to formula~\eqref{deri} resolves the
centuries-old paradox famously criticized by George Berkeley \cite{Be}
(in terms of the {\em Ghosts of departed quantities},
cf.~\cite[Chapter~6]{St}), and provides a rigorous basis for the
calculus.

\subsection{Continuity}
\label{106}
A function~$f$ is continuous at~$x$ if the following condition is
satisfied:~$y\approx x$ implies~$f(y)\approx f(x)$.

\subsection{Uniform continuity} 
\label{107}
A function~$f$ is uniformly continuous on~$I$ if the following
condition is satisfied:

\begin{itemize}
\item
standard: for every~$\epsilon>0$ there exists a~$\delta>0$ such that
for all~$x\in I$ and for all~$y\in I$, if~$|x-y|<\delta$ then
$|f(x)-f(y)| < \epsilon$.
\item
non-standard: for all~$x\in I^*$, if~$x\approx y$ then~$f(x) \approx
f(y)$.
\end{itemize}


\subsection{Hyperinteger}
\label{108}
A hyperreal number~$H$ equal to its own integer part 
\[
H = [H]
\]
is called a hyperinteger (here the integer part function is the
natural extension of the real one).  The elements of the complement
$\Z^* \setminus \Z$ are called infinite hyperintegers.

\subsection{Proof of extreme value theorem}
\label{109}
Let~$H$ be an infinite hyperinteger.  The interval~$[0,1]$ has a
natural hyperreal extension.  Consider its partition into~$H$
subintervals of equal length~$\frac{1}{H}$, with partition points~$x_i
= i/H$ as~$i$ runs from~$0$ to~$H$.  Note that in the standard
setting, with~$n$ in place of~$H$, a point with the maximal value
of~$f$ can always be chosen among the~$n+1$ partition points~$x_i$, by
induction.  Hence, by the transfer principle, there is a
hyperinteger~$i_0$ such that~$0\leq i_0 \leq H$ and
\begin{equation}
\label{101b}
f(x_{i_0})\geq f(x_i) \quad \forall i= 0,...,H.
\end{equation}
Consider the real point
\begin{equation*}
c= {\rm st}(x_{i_0}).
\end{equation*}
An arbitrary real point~$x$ lies in a suitable sub-interval of the
partition, namely~$x\in [x_{i-1},x_i]$, so that~${\rm st}(x_i) = x$.
Applying ``st'' to the inequality \eqref{101b}, we obtain by
continuity of~$f$ that~$f(c)\geq f(x)$, for all real~$x$, proving~$c$
to be a maximum of~$f$ (see \cite[p.~164]{Ke} and \cite[Chapter~12,
p.~324]{Eb}).

\subsection{Limit}
\label{1010}
We have~$\lim_{x\to a} f(x) = L$ if and only if whenever the
difference~$x-a$ is infinitesimal, the difference~$f(x)-L$ is
infinitesimal, as well, or in formulas: if~${\rm st}(x)=a$ then~${\rm
st}(f(x)) = L$.

Given a sequence of real numbers~$\{x_n|n\in \mathbb{N}\}$, if~$L\in
\mathbb{R}\;$ we say~$L$ is the limit of the sequence and write~$L =
\lim_{n \to \infty} x_n$ if the following condition is satisfied:
\begin{equation}
\label{102b}
{\rm st} (x_H)=L \quad \mbox{\rm for all infinite } H
\end{equation}
(here the extension principle is used to define~$x_n$ for every
infinite value of the index).  This definition has no quantifier
alternations.  The standard~$(\epsilon, \delta)$-definition of limit,
on the other hand, does have quantifier alternations:
\begin{equation}
\label{disaster}
L = \lim_{n \to \infty} x_n\Longleftrightarrow \forall \epsilon>0\;,
\exists N \in \mathbb{N}\;, \forall n \in \mathbb{N} : n >N \implies
d(x_n,L)<\epsilon.
\end{equation}

\section*{Acknowledgment}

We are grateful to H. J. Keisler for checking an earlier version of
the non-standard argument and pointing out a gap.


\begin{thebibliography}{AI}


\bibitem{AK} Ambrosio, L.; Katz, M.: Flat currents modulo~$p$ in
metric spaces and filling radius inequalities, preprint.

\bibitem{Be} Berkeley, George: The Analyst, a Discourse Addressed to
an Infidel Mathematician (1734).


\bibitem{Eb} Ebbinghaus, H.-D.; Hermes, H.; Hirzebruch, F.; Koecher,
M.; Mainzer, K.; Neukirch, J.; Prestel, A.; Remmert, R.: Numbers.
With an introduction by K. Lamotke.  Translated from the second 1988
German edition by H. L. S. Orde.  Translation edited and with a
preface by J. H. Ewing.  {\em Graduate Texts in Mathematics},
\textbf{123}.  Readings in Mathematics. Springer-Verlag, New York,
1991.


\bibitem{Gr1} Gromov, M.: Filling Riemannian manifolds.  {\em
J. Diff. Geom.}, \textbf{18} (1983), 1--147.


\bibitem{Gu2} Guth, L.: Notes on Gromov's systolic estimate.  {\em
Geom. Dedicata}, \textbf{123} (2006), 113--129.


\bibitem{KSh} Kanovei, V.; Shelah, S.: A definable nonstandard model
of the reals. {\em J. Symbolic Logic\/} \textbf{69} (2004), no.~1,
159--164.


\bibitem{SGT} Katz, M.: Systolic geometry and topology.  With an
appendix by Jake P. Solomon.  {\em Mathematical Surveys and
Monographs}, \textbf{137}.  American Mathematical Society, Providence,
RI, 2007.


\bibitem{Ke} Keisler, H. Jerome: Elementary Calculus: An Infinitesimal
Approach.  Second Edition.  Prindle, Weber \& Schimidt, Boston, '86.


\bibitem{Li} Lightstone, A. H.: Infinitesimals.  {\em
Amer. Math. Monthly\/} \textbf{79} (1972), 242--251.



\bibitem{Ro66} Robinson, Abraham: Non-standard analysis. North-Holland
Publishing Co., Amsterdam 1966.


\bibitem{Ro} Robinson, Abraham: Non-standard analysis.  Reprint of the
second (1974) edition. With a foreword by Wilhelmus A. J. Luxemburg.
Princeton Landmarks in Mathematics. Princeton University Press,
Princeton, NJ, 1996.


\bibitem{St} Stewart, I.: From here to infinity.  A retitled and
revised edition of The problems of mathematics [Oxford Univ. Press,
New York, 1992].  With a foreword by James Joseph Sylvester.  The
Clarendon Press, Oxford University Press, New York, 1996.

\end{thebibliography}
\end{document}